\documentclass[CRMATH,Unicode,manuscript]{cedram}

\usepackage{longtable}

\title{Approximation of the number of nonunimodular zeros of a self-reciprocal polynomial}

\author{\firstname{Dragan} \lastname{Stankov}}
\address{Katedra Matematike RGF-a,
Faculty of Mining and Geology,
University of Belgrade,
Belgrade, \DJ u\v sina 7,
Serbia}
\thanks{The author received partial support from  
Ministry of Educations, Science and Technological Development of Republic of Serbia Project \#174032.}
\email[D. Stankov]{dstankov@rgf.bg.ac.rs}



\begin{abstract} 
 We introduce the ratio of the number of roots, not equal to 1 in modulus, of a reciprocal
 polynomial $R_d(x)$ to its degree $d$. For some sequences of reciprocal polynomials we show that the ratio has a
 limit $L$ when $d$ tends to infinity. Each of these sequences is defined using a two-variable polynomial $P(x,y)$ 
 so that $R_d(x) = P(x,x^n)$. 
 We present a few methods for approximation of the  limit ratio, some of them we originally developed. 
 In a previous paper we have calculated the exact value of the limit ratio of polynomials, correlated
 to many bivariate polynomials, of degree two in $y$, having small Mahler measure listed by Boyd and Mossinghoff. 
 Now we approximate the limit ratio of $P(y^n,y)$, for all polynomials in the list. Knowing some exact values of the limit ratio, we compare the error of the approximation and the time duration of these methods. 
 We show that the limit ratio of the sequence $P(x,x^n)$ usually is not equal to the limit ratio of
 the sequence $P(y^n,y)$, unlike the Mahler measure.
\end{abstract}

\begin{document}

\maketitle

\section{Introduction}

    
For a polynomial in one variable 
\begin{equation}\label{eq:Poly1}
P(x) = a_d \prod_{j=1}^{d}(x - \alpha_j)\in \mathbb{C}_d[x]
\end{equation}
($a_d \ne 0$), having zeros $\alpha_1,\alpha_2,\ldots,\alpha_d$, its 
Mahler measure is defined $$M(P(x)) = |a_d|\prod_{j=1}^{d}\max(1, |\alpha_j|).$$

The definition of the Mahler measure could be extended to polynomials in several
variables. We recall 
Jensen's formula, which states that
$\int_0^1 \log |P(e^{2\pi i\theta})| d\theta= \log |a_d| + \sum_{j=1}^d \log\max (|\alpha_j|, 1)$
Thus
$$M(P) = \exp \left\{ \int_0^1 \log |P(e^{2\pi i\theta})|d\theta\right\},$$
so $M(P)$ is just the geometric mean of $|P(z)|$ on the torus $T$.
Hence a natural
candidate for $M(P(x_1,x_2,\ldots,x_n)$ is
$$M(P) = \exp\left\{\int_0^1 \cdots\int_0^1\log |P(e^{2\pi i\theta_1},\ldots,e^{2\pi i\theta_n})|d\theta_1\cdots d\theta_n \right\}.$$
Various properties of Mahler measure have been investigated in \cite{Dub, LNS, BGMP, GL, BZ}.

The smallest known Mahler measure greater than 1 of a polynomial with integer coefficients is $M(x^{10} + x^9 - x^7 - x^6 - x^5 - x^4 - x^3 + x + 1) = 1.17628081 \ldots$.
This still stands as the smallest value of $M(P) > 1$, in spite of extensive computation
done since 1933 by many mathematicians.

The 
smallest known Mahler measures in 
two variables are
$$M((x + 1)y^2 + (x^2 + x + 1)y + x(x + 1)) = 1.25542\ldots$$ and
$$M(y^2 + (x^2 + x + 1)y + x^2) = 1.28573\ldots$$
The theorem of Boyd and Lawton (1981) claims that
\begin{equation}\label{BLF}
M(P (x, x^n))\rightarrow M(P(x,y))
\end{equation}
as $n\rightarrow \infty$.

Pritsker \cite{Pri} defined a natural areal analog of the Mahler measure and studied its properties. Flammang \cite{F1a1} introduced the absolute S-measure for the polynomial \eqref{eq:Poly1} defined by
$$s(P) :=\frac{1}{d}\sum_{j=1}^{d}|{\alpha}_j|,$$ as an analog to the Mahler measure and studied its properties. Our idea is neither to multiply nor to sum moduli of the zeros outside of the unit disc but only to count them. The problem of counting the 
zeros that a polynomial has on the unit circle is still an open problem (see \cite{BCFJ, BEFL, GV,Dru, Muk,VIEIRA2021113169, DAS2024106087, FV}). 

\noindent
Let $I(P)$ denote the number of complex
zeros of $P(x)$, which are 
$<1$ in modulus, counted with multiplicities. Such zeros 
are called internal. 

\noindent
Let $U(P)$ denote the number of zeros of $P(x)$, which are 
$=1$ in modulus,
(again, counting with multiplicities). Such zeros 
are called unimodular.

\noindent
Let $O(P)$ denote the number of complex
zeros of $P(x)$, which are 
$>1$ in modulus, counted with multiplicities. Such roots we call external roots.
Then clearly 

\noindent
$I(P)+U(P)+O(P)=d$, $I(P_1P_2)=I(P_1)+I(P_2) $, $U(P_1P_2)=U(P_1)+U(P_2) $, $O(P_1P_2)=O(P_1)+O(P_2) $.

\noindent
Pisot number can be defined as a real algebraic integer greater than 1 having the minimal polynomial $P(x)$ of degree $d$ such that 
$I(P)=d-1$.

\noindent
Salem number
is a real algebraic integer $> 1$ having the minimal polynomial $P(x)$ of degree $d$ such that 
$U(P)=d-2$, $I(P)=1$.

We say that a polynomial of degree $d$ is self-reciprocal or reciprocal, 
if $P(x) = x^dP(1/x)$.
Clearly,
if $\alpha_j$ is a root of $P(x)$, then $1/\alpha_j$ is also a root of $P(x)$ so that $I(P)=O(P)$.
The minimal polynomial of a Salem number is a self-reciprocal polynomial.

Let $C(P)=\frac{I(P)+O(P)}{d}$ be the ratio of the number of nonunimodular zeros of $P$ to its degree. Actually, it is the probability that a randomly chosen zero is not unimodular. 
Since the numbers of the internal and the external roots are equal, it follows that $C(P)=\frac{2O(P)}{d}$. 
In \cite{Sta1} we proved for a class of reciprocal, bivariate polynomials $P(x,y)$, having degree two in $y$, that the limit 
\begin{equation}\label{LC}
  \lim_{n\rightarrow \infty} C(P(x,x^n))
\end{equation}
exists. So it is reasonable to conjecture that the limit exists for any $P(x,y)$. In \cite{BM} Boyd and Mossinghoff, in their Chapter 6, have given some observations supporting the conjecture. It is not our aim here to prove the conjecture. For some polynomials we determine the explicit value of the limit. For many others we approximate the limit using various methods, some of them we have originally developed.    
We are going to prove here the following 

\begin{theorem} \label{MainTh} 
If 
\begin{equation}\label{P23inv}
  P^{\times}_{2,3}(x,y)= 1+y+xy+xy^2+xy^3+x^2y^3+x^2y^4
\end{equation}
then
  $$\lim_{n\to \infty}C(P^{\times}_{2,3}(x,x^n))=\frac{1}{\pi}\arccos\frac{ 3}{4}\approx 0.230053456. $$
\end{theorem}
If the limit in \eqref{LC} exists, it is useful to denote it with $LC(P)$.  


\section{The methods of the approximation}

\subsection{Definitions}

In \cite{BM} Boyd and Mossinghoff have introduced the following definition of certain families of polynomials, quadratic in $y$ used in Table \ref{tab:table2} at the end of this section.

\begin{definition}\label{def:PQRST}

%
    \begin{tabular}{l c l l}
      $P_{a,b}(x,y)$ &=&  $x^{\max(a-b,0)}\left(\sum_{j=0}^{a-1}x^j+\sum_{j=0}^{b-1}x^jy+x^{b-a}\sum_{j=0}^{a-1}x^jy^2\right)$ \\
      $Q_{a,b}(x,y)$  &=&  $x^{\max(a-b,0)}(1 + x^a + (1 + x^b)y + x^{b-a}(1 + x^a)y^2$ \\
      $R_{a,b}(x,y)$  &=&  $x^{\max(a-b,0)}(1 + x^a + (1 - x^b)y - x^{b-a}(1 + x^a)y^2$ \\
      $S_{a,b,\epsilon}(x,y)$  &=&  $1 + (x^a + \epsilon)(x^b + \epsilon)y + x^{a+b}y^2$, $\epsilon=\pm 1$ \\
      $T(x, y)$ &=& $yf(x) + f^*(x)$ \\ 
    \end{tabular}
\end{definition}


\begin{definition}\label{def:switch}
If we switch variables in a bivariate polynomial $P(x,y)$, we get a bivariate polynomial, which we call an inverted polynomial and denote $P^{\times}(x,y)$ so that 
$$P^{\times}(x,y):=P(y,x).$$ 
\end{definition}

For example, it follows from Definition \ref{def:PQRST}. that 
$P_{2,3}(x,y)=1+x+yx+yx^2+yx^3+y^2x^3+y^2x^4$ that is quadratic in $y$. 
Then $P^\times_{2,3}(x,y)=1+y+xy+xy^2+xy^3+x^2y^3+x^2y^4$ is quartic in $y$.
The next example explains the correlation between rows 2 and 2' in Table \ref{tab:table2}:
as $P_{2,1}(x,y)=1+x+yx+y^2x+y^2x^2$, it follows that 
$P^\times_{2,1}(x,y)=1+y+xy+x^2y+x^2y^2=P_{1,3}(x,y).$ 

\subsection{The method of Boyd and Mossinghoff (the BM method)}

In section 6 of \cite{BM}, Boyd and Mossinghoff introduced the following definitions: let $P(x,y)$ be of degree $g$ in $y$, let $\nu(x)$
denote the number of roots of $P(x, y) = 0$ with $|y(x)|>1$. Define $\delta=\delta(P):=(1/g)\int_0^1\nu(\exp(2\pi i t))dt$. Then the number of zeros of $P(x,x^n)$ outside the unit circle is asymptotically $\sim \delta g n$. They have calculated $\delta(P_{2,3}) = .06640475\ldots$, which is in accordance with our result in \cite{Sta1} $LC(P_{2,3})=0.1328095\ldots$ because $LC(P)=2\delta(P)$ is valid. 
In that way we get a formula for approximation of the limit ratio (BM method):
\begin{equation}\label{APP2}
LC(P)=\frac{2}{g}\int_0^1\nu(e^{2\pi i t})dt.  
\end{equation}

We present the Pari/GP code for calculating $LC(P^\times(k,m))$ using BM method: 

\begin{verbatim}
\p100
P1(k, m, s)=
{ v=vector(2*k-2+m);
  for(i = 1, k, v[2*k-1+m-i]=1);
  for(i = 1, k, v[i]=v[i]+exp(4*I*Pi*s));
  for(i = k, k+m-1, v[i]=v[i]+exp(2*I*Pi*s));  
  return(v);
}
Pr(k,m,s)=
{
r=polroots(Pol(P1(k,m,s)));
w=matsize(r)[1];
b=0;
for(i=1,w,if( abs(r[i]) >1.000000001,b=b+1);
);
return(b/w)
}
4*intnum(s=0,0.5,abs(Pr(2,3,s)),2)
\end{verbatim}

\subsection{The MBM method - our modification of the BM method}

We can notice that the integrand in \eqref{APP2}, $\nu(e^{2\pi i t})$, is a step function on $[0,1]$. Thus we introduce a modification of the BM method that significantly accelerates the calculation of the integral. We need to determine $t_0=0<t_1<\cdots<t_m=1$ such that the integrand $\nu$ is a constant function on $(t_{i-1},t_{i})$ and $t_1,t_2,\ldots,t_{m-1}$ are jump points, i.e., values of the integrand in $t_i-\epsilon$ and $t_i+\epsilon$ are different for all $\epsilon > 0$. To do this we start with a partition $I_j=[(j-1)/n,j/n], j=1,2,\ldots,n$. We investigate whether values of the integrand of the left and the right endpoint of $I_j$ are different. If they are, we use the bisection method to determine the jump point in $I_j$. Of course, $n$ should be sufficiently large so that there is at most one jump point in each $I_j$. Finally, the formula for approximation of the limit ratio (using the MBM method) is: 
\begin{equation}\label{APP3}
LC(P)=\frac{2}{g}\sum_{i=1}^m(t_{i}-t_{i-1})\nu\left(e^{2\pi i (t_{i-1}+t_{i})/2}\right).  
\end{equation}

\subsection{The CAP method - our method based on Cauchy’s argument principle}

The following theorem, known as Cauchy's argument principle, is useful for determination $C(P)$.
\begin{theorem}
If $f(x)$ is a meromorphic function inside and on some closed, simple contour $K$, and $f$ has no zeros or poles on $K$, then
$$\frac{1}{2\pi i}\oint_K \frac{f'(x)}{f(x)}dx=Z-\Pi\;, $$
where $Z$ and $\Pi$ denote, respectively, the number of zeros and poles of $f(x)$ inside the contour $K$, with each zero and pole counted as many times as its multiplicity and order, respectively, indicate.
\end{theorem}
If $f(x)$ is a polynomial $P(x)$ and the contour is a cycle concentric with the unit circle, then $\Pi=0$ and we can introduce the substitution $x=r\exp(2\pi i t)$ so that 
$$Z=\frac{1}{d} \int_0^1 \frac{P'(re^{2\pi i t})re^{2\pi i t}}{P(re^{2\pi i t})}dt .$$
Since $I(P)$ is a finite number, we can choose $r<1$, $r$ close to 1, such that all roots of $P$, less than 1 in modulus, are settled in the interior of the circle $\rho=re^{2\pi i t}$. 
If $P(x)$ is reciprocal, the number of zeros of $P(x)$ inside the unit circle is equal to the number of zeros of $P(x)$ outside the unit circle so that
\begin{equation}\label{CPappr}
C(P)=\frac{2}{d} \int_0^1 \frac{P'(re^{2\pi i t})re^{2\pi i t}}{P(re^{2\pi i t})}dt.
\end{equation}

Let $P(x,y)$ be written $$P(x,y)=a_{g}(x)y^{g}+a_{g-1}(x)y^{g-1}+ \cdots + a_0(x).$$ The degree $d$ of $P(x,x^n)$ is
\begin{equation}\label{eq:degree1}
d=\deg(a_{g}(x))+ng.
\end{equation}
Using \eqref{CPappr} there is $r<1$ such that
$$ C(P(x,x^n))=\frac{2}{d}\int_0^1 \frac{\frac{\partial}{\partial x}P(re^{2\pi i t},r^ne^{2\pi i nt})re^{2\pi i t}+n\frac{\partial}{\partial y}P(re^{2\pi i t},r^ne^{2\pi i nt})r^ne^{2\pi i nt}}{P(re^{2\pi i t},r^ne^{2\pi i nt})}dt.$$
We can split the previous integral into two integrals and replace $d$ using \eqref{eq:degree1}. The first integral
$$ \frac{2}{\deg(a_{g}(x))+ng} \int_0^1 \frac{\frac{\partial}{\partial x}P(re^{2\pi i t},r^ne^{2\pi i nt})re^{2\pi i t}}{P(re^{2\pi i t},r^ne^{2\pi i nt})}dt \to 0,\;\; n \to \infty.$$
In the second integral
$$\frac{2}{\deg(a_{g}(x))+ng} \int_0^1 \frac{n\frac{\partial}{\partial y}P(re^{2\pi i t},r^ne^{2\pi i nt})r^ne^{2\pi i nt}}{P(re^{2\pi i t},r^ne^{2\pi i nt})}dt $$ 
we can see that $$\frac{2n}{\deg(a_{g}(x))+ng}\to \frac{2}{g}$$ as $n$ tends to infinity.
The term $e^{2\pi i t}$ becomes increasingly uncorrelated with $e^{2\pi i nt}$ as $n \to \infty$, so that $nt$ can be replaced with a variable $s$ that vary independently to $t$:
$$ \lim_{n \to \infty} \frac{2}{g} \int_0^1 \frac{\frac{\partial}{\partial y}P(re^{2\pi i t},r^ne^{2\pi i nt})r^ne^{2\pi i nt}}{P(re^{2\pi i t},r^ne^{2\pi i nt})}dt \approx \lim_{n \to \infty} \frac{2}{g}  \int_0^1 \int_0^1 \frac{\frac{\partial}{\partial y}P(re^{2\pi i t},r^ne^{2\pi i s})r^ne^{2\pi i s}}{P(re^{2\pi i t},r^ne^{2\pi i s})}dt ds.$$
Finally, it is clear why we introduce the following formula (CAP method) 
\begin{equation}\label{eq:doubleint}
LC(P)\approx \lim_{
\scalebox{0.7}{
\begin{tabular}{l}
$r\to 1-0$  \\
$n\to \infty$\\
\end{tabular}}
}
\frac{2}{g} \int_0^1 \int_0^1 \frac{\frac{\partial}{\partial y}P(re^{2\pi i t},r^ne^{2\pi i s})r^ne^{2\pi i s}}{P(re^{2\pi i t},r^n e^{2\pi i s})}dt ds.
\end{equation}
when $P(x,x^n)$ has unimodular roots. 
Our calculations show that for $n=300$, $r=0.99999$ the double integral formula on the right of \eqref{eq:doubleint} gives the value close to the exact value (the error is
$<10^{-5}$). We should be careful not to lose the accuracy if we take in \eqref{eq:doubleint} larger values of $n$ and $r$ closer to 1.





We present formulae suitable for programming of the CAP-method for $P_{k,m}(x,y)$ and for $P^\times_{k,m}(x,y)$.
According to the Definition \ref{def:PQRST}. $P_{k,m}(x,y)$ is
$$P(k,m,x,y):=\sum_{j=0}^{k-1}x^j+y\sum_{j=k-1}^{k+m-2}x^j +y^2\sum_{j=k+m-2}^{2k+m-3}x^j.$$
The partial derivative of $P_{k,m}(x,y)$ with respect to $x$ is 
$$Px(k,m,x,y):=\sum_{j=1}^{k-1}jx^{j-1}+y\sum_{j=k-1}^{k+m-2}jx^{j-1} +y^2\sum_{j=k+m-2}^{2k+m-3}jx^{j-1}.$$
The partial derivative of $P_{k,m}(x,y)$ with respect to $y$ is 
$$Py(k,m,x,y):=0+1\sum_{j=k-1}^{k+m-2}x^j +2y\sum_{j=k+m-2}^{2k+m-3}x^j.$$

Then $LC(P_{k,m}(x,y))$ can be approximated using 
$$Cy(k,m,n,r):=\frac{2}{2}\int_{0}^{1}\int_{0}^{1}\frac{Py(k,m,re^{2i\pi t},r^n e^{2i\pi s})(r^n e^{2i\pi s})}{P(k,m,r e^{2i\pi t},r^n e^{2i\pi s})}dtds,$$
and $LC(P^\times_{k,m}(x,y))$ can be approximated using
$$Cx(k,m,n,r):=\frac{2}{2k+m-3}\int_{0}^{1}\int_{0}^{1}\frac{Px(k,m,r^ne^{2i\pi s},r e^{2i\pi t})(r^n e^{2i\pi s})}{P(k,m,r^n e^{2i\pi s},r e^{2i\pi t})}dtds,$$
where $r$ is close to 1 and $r<1$, $n\ge 200$.

\subsection{The table}

In Table \ref{tab:table2} we present limit points calculated in \cite{BM} of the Mahler measure of bivariate reciprocal polynomials $P(x,y)$ 
in ascending order. 
Then we present in Table \ref{tab:table2} the limit points $LC(P(x,y))$ of the ratio of the number of nonunimodular roots of the polynomial $P(x,x^n)$ to its degree when $n \rightarrow \infty$. For $P(x,y)$ that are quadratic in $y$, we use the exact values that we have determined in \cite{Sta1}. In the last column of Table \ref{tab:table2}, we complete the calculation by the limit points $LC(P^{\times}(x,y))$. As in \cite{BM} polynomials $P_{a,b}(x,y)$, $Q_{a,b}(x,y)$, $R_{a,b}(x,y)$, $S_{a,b,\epsilon}(x,y)$, defined in Definition \ref{def:PQRST}, are labeled as $P(a,b)$, $Q(a,b)$, $R(a,b)$, $S(a,b,\textrm{sgn}(\epsilon))$ respectively, in Table \ref{tab:table2}. Some polynomials are identified by the sequences; for example, the third smallest known limit point $(1 + x) + (1 - x^2 + x^4)y + (x^3 + x^4)y^2$, is identified by [++000, +0$-$0+, 000++], as in \cite{BM}. Polynomials in Table \ref{tab:table2} are written explicitly in Table D.2 of \cite{MS}. 

    \begin{longtable}{r l l l l}
    \caption{ Limit points of Mahler measure and limit points of the ratio of number of nonunimodular roots of a polynomial to its degree.\label{tab:table2}}\\
       & $\lim\limits_{n\rightarrow\infty}M(P(x,x^n))$ & $P$ & $\lim\limits_{n\rightarrow\infty}C(P(x,x^n))$ & $\lim\limits_{n\rightarrow\infty}C(P^\times(x,x^n))$  \\
      \hline
1. & 1.2554338662666087457 & P(2, 3)& 0.1328095098966884 & 0.230053456162615 \\ 
2. & 1.2857348642919862749 & P(2, 1)& 0.1608612465103325 & 0.333333333333333 \\ 
2'. & 1.2857348642919862749 & P(1, 3)& 0.3333333333333333 & 0.160861246510332 \\ 
3. & 1.3090983806523284595 & [++000, & 0.2970136797597501 & 0.097583122975771 \\ 
& & \multicolumn{2}{l}{+0$-$0+, 000++]} & \\
4. & 1.3156927029866410935 & P(3, 5)& 0.1646453474320021 & 0.261925575993535 \\ 
5. & 1.3247179572447460260 & T($1 + x - x^3$) & 0 & 0.132322561324637  \\
6. & 1.3253724973075860349 & P(3, 4)& 0.1739784246485862 & 0.288589141585482 \\
7. & 1.3320511054374193142 & P(2, 5)& 0.2634504964561481 & 0.272862064298000 \\
8. & 1.3323961294587154121 & S(1, 3,+)& 0.3814904582918582 & 0.124091876652015 \\ 
9. & 1.3381374319388410775 & P(3, 2)& 0.1871346248477649 & 0.345086459236550 \\
10. & 1.3399999217381835332 & P(4, 7)& 0.1784746137157699 & 0.273841185766290 \\ 
11. & 1.3405068829308471079 & P(3, 1)& 0.1895159205822178 & 0.368337855217854 \\ 
12. & 1.3497161046696958653 & T($1 + x^2 - x^7$) & 0 & 0.130905750935710 \\ 
13. & 1.3500148321630142650 & P(3, 7)& 0.2403097841316317 & 0.282722388246059 \\ 
14. & 1.3503169790598690950 & S(1, 4,$-$)& 0.3105668890134219 & 0.097319162141083 \\ 
15. & 1.3511458956697046903 & P(4, 5)& 0.1902698620670582 & 0.311233156483764 \\
16. & 1.3524680625188602961 & P(5, 9)& 0.1860703555283188 & 0.279768767163400 \\
17. & 1.3536976494626355711 & Q(1, 6)& 0.1893226580984896 & 0.079331472814232 \\
18. & 1.3567481051456008311 & P(4, 3)& 0.1964065801899085 & 0.347838496792791 \\
19. & 1.3567859884526454967 & P(5, 8)& 0.1908351326172760 & 0.293851334230770 \\
20. & 1.3581296324044179208 & [++00000, & 0.3755212901021780 & 0.107925225247525 \\ 
& & \multicolumn{2}{l}{+0$-$$-$$-$0+, 00000++]} & \\
21. & 1.3585455903960511404 & P(4, 1)& 0.1981783524823832 & 0.376084355688991 \\
22. & 1.3592080686995589268 & P(4, 9)& 0.2295536290347317 & 0.285604424375482 \\
23. & 1.3598117752819405021 & P(6, 11)& 0.1908185635976727 & 0.283185962099926 \\
24. & 1.3598158989877492950 & S(1, 6,+)& 0.3638326121576760 & 0.080362533690731 \\
25. & 1.3599141493821189216 & T($1 + x + x^8$) & 0 & 0.062172551844474 \\
26. & 1.3602208408592842371 & P(5, 7)& 0.1947758787175794 & 0.307985887166100 \\
27. & 1.3627242816569882815 & P(5, 6)& 0.1976969967166677 & 0.321914094334985 \\
28. & 1.3636514981864992177 & S(3, 5,+)& 0.3616163835316277 & 0.177841235500398 \\
29. & 1.3641995455827723418 & T($1 - x^2 + x^5$) & 0 & 0.178772346520853 \\
30. & 1.3644358117806362770 & [+000,00++, & 0.3504700257823537 & 0.169413093518251 \\
& & \multicolumn{2}{l}{++00, 000+]} & \\
31. & 1.3645459857899151366 & P(7, 13)& 0.1940425569464528 & 0.285345672159789 \\
32. & 1.3646557293930641449 & P(5, 11)& 0.2236027778291241 & 0.286902784448591 \\
33. & 1.3650623157174417179 & S(2, 7,$-$)& 0.3360946113639976 & 0.115525164633522 \\ 
34. & 1.3654687370557201592 & P(5, 4)& 0.2007692138817449 & 0.348374180979957 \\
35. & 1.3659850533667936783 & [++000,++0$-$0, & 0.2069305454044983 & 0.206930545404498 \\ 
& & \multicolumn{2}{l}{ 00000,0$-$0++, 000++]} & \\
36. & 1.3661459663116649518 & P(5, 3)& 0.2014521139875612 & 0.359293353026221 \\
37. & 1.3665709746056369455 & P(5, 2)& 0.2018615118309531 & 0.371006144584871 \\
38. & 1.3668078899273126149 & P(5, 1)& 0.2020844014923849 & 0.378452305048962 \\
39. & 1.3668830708592258921 & R(1, 5)& 0.1417550822341309 & 0.126211051843860 \\
40. & 1.3669909125179202255 & P(7, 12)& 0.1970232013102869 & 0.294801531511566 \\
41. & 1.3677988580117157740 & P(8, 15)& 0.1963614081210482 & 0.286799704864039 \\
42. & 1.3678546316653002345 & T($1 + x^4 + x^{11}$) & 0 & 0.172252351901681 \\ 
43. & 1.3681962517212729703 & P(6, 13)& 0.2199360577499605 & 0.287642585167356 \\
44. & 1.3682140096679950123 & P(1, 9)& 0.2082012946810569 & 0.066657322721448 \\
45. & 1.3683434385467330804 & [++00000, & 0.3045732337814742 & 0.213131613170404 \\
& & \multicolumn{2}{l}{++0$-$0++,00000++]} & \\
46. & 1.3687474425069274154 & P(6, 7)& 0.2014928273535877 & 0.327637984546821 \\
47. & 1.3689491694959833864 & P(7, 11)& 0.1994880038265199 & 0.304157343580054 \\
48. & 1.3697823199880122791 & S(1, 9,+)& 0.3622499773114010 & 0.059018757923146 \\
\end{longtable}
\section{The proof of the Theorem \ref{MainTh}}
Boyd and Mossinghoff \cite{BM} have shown that $P_{2.3}$ has the smallest known Mahler measure of irreducible dimension-2 polynomials. We showed in \cite{Sta1} that it has the smallest $LC$ among all bivariate reciprocal polynomials having degree two in $y$ in their table. We have determined the exact value $LC(P_{2.3})=1-\frac{2}{\pi}\arccos\left(\frac{\sqrt{2}}{2}-\frac{1}{2}\right)\approx 0.1328095\ldots$. 
We can show by changing the order of integration in a double integral that
switching of the variables in a bivariate polynomial does not affect to the Mahler
measure, i.e., $M(P(x, y)) = M(P^{\times}(x,y))$. It follows from the Boyd-Lawton formula \eqref{BLF} that $\lim_{n\to \infty}M(P(x,x^n))=\lim_{n\to \infty}M(P^{\times}(x,x^n)).$ 
Our task in the proof of the Theorem \ref{MainTh} is to determine exact value of $LC(P^{\times}_{2,3}(x,x^n)).$ We can notice that this value is not equal to $LC(P_{2,3}(x,x^n)).$ 
In the proof we will need the curve that touches each member of a given family of curves, so called the envelope of that family.

\subsection{Envelope}
\begin{Definition}
The family of curves 
\begin{equation}\label{famCur}
  f(x,y,\alpha)
\end{equation}
has an envelope
\begin{equation}\label{Env}
x=g(\alpha), y=h(\alpha)
\end{equation}
if, and only if, for each $\alpha = \alpha_0$ the point $(g(\alpha_0), h(\alpha_0))$ of the curve \eqref{Env} lies on the curve $f(x, y, \alpha_0) = 0$ and both curves have the same tangent line there.
\end{Definition}
We present Theorem 6. in Chapter 4. of \cite{widder1989advanced} along with its short proof.
\begin{theorem} If

  1. $f(x,y,\alpha)$, $g(\alpha)$, $h(\alpha)\in C^1$,
  
  2. $f_x'^2+f_y'^2\ne 0$
  
  3. $g'^2+h'^2\ne 0$,
  
  4. $f(g(\alpha),h(\alpha),\alpha)\equiv0$
  
  5. $f'_{\alpha}(g(\alpha),h(\alpha),\alpha)\equiv0$

then family \eqref{famCur} has curve \eqref{Env} as an envelope.
\end{theorem}
\begin{proof}
For each $\alpha$ the point $(g(\alpha), h(\alpha))$ lies on the curve \eqref{famCur} by hypothesis 4. For each $\alpha$ the slope of the curve \eqref{famCur} is
  \begin{equation}
    \begin{cases}
      \frac{dy}{dx}=-\frac{f'_x}{f'_y}, & \text{if}\ f'_y\ne 0, \\
      \infty, & \text{othervise.}
    \end{cases}
  \end{equation}
Differentiating 4. partially with respect to $\alpha$ and using 5. we have,
$$f'_xg' + f'_yh' + f'_{\alpha} = f'_xg' + f'_yh' = 0.$$
Hence,
  \begin{equation}\label{slope}
    \begin{cases}
-\frac{f'_x}{f'_y}=\frac{h'}{g'}, & \text{if}\ g'\ne 0, \\
      \infty, & \text{othervise.}
    \end{cases}
  \end{equation}
Since the right-hand side of \eqref{slope} is precisely the slope of the curve \eqref{famCur}, the proof is complete. It is clear that when $f'_y$ vanishes $f'_x$ does not and that then $g'$ must also vanish. Both slopes are then infinite.
\end{proof}

This theorem provides a simple method of determining the functions $g$ and $h$. We have only to solve the equations
\begin{equation}\label{eqEnv}
  f(x,y,\alpha)=0,\;\; f'_{\alpha}(x,y,\alpha)=0
\end{equation}
as simultaneous equations in $x$ and $y$. The equations \eqref{eqEnv} are called parametric equations of the envelope and eliminating of $\alpha$ between them produces the equation of envelope.

\begin{proof} \textbf{of Theorem \ref{MainTh}}
To find the unimodular roots of $P^\times_{2,3}(x,x^n) $ we have to substitute $ x=e^{it}$ $y=e^{int}$ into $P^\times_{2,3}(x,y)$. Then we have to solve $P^\times_{2,3}(e^{it},e^{int})=0 $. 
If we rewrite \eqref{P23inv} as $$P^{\times}_{2,3}(x,y)=xy^2\left( xy^2+xy+y+1+\frac{1}{y}+\frac{1}{xy}+\frac{1}{xy^2}\right)$$
we get the equation $f_1(t,n)=0$ where 
$$f_1(t,n)=2\cos((2n+1)t)+2\cos((n+1)t)+2\cos(nt)+1 .$$ Each unimodular root correspondents to an intersection point of $f_1(t,n)$ and $t$-axes. To estimate the number of the unimodular roots in different parts of $[0,2\pi]$, we need to determine envelopes of the family of curves $z=f_1(t,n)$. The envelopes are presented on Figure \ref{fig:1} by solid blue, green and brown curves.

\begin{figure}[t]
  \includegraphics[width=0.9\textwidth, height=0.5\textheight]{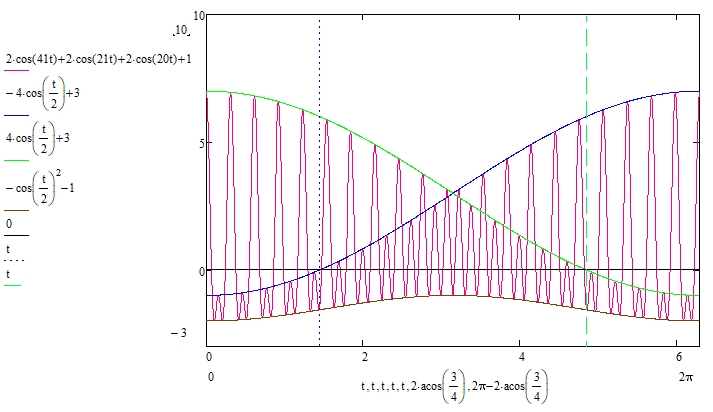}
\caption{Estimation of the number of nonunimodular roots in different subsegments of $[0,2\pi]$}
\label{fig:1}       
\end{figure}

The system of the equations \eqref{eqEnv} become here $f(t,z,n)=f_1(t,n)-z=0$ and $f'_n(t,z,n)=\frac{\partial}{\partial n}f_1(t,n)=0$ i.e.
\begin{equation}\label{eqEnv1}
z=2\cos((2n+1)t)+2\cos((n+1)t)+2\cos(nt)+1\;\;\;
\end{equation}
and 
\begin{equation}\label{eqEnv2}
0=-4t\sin((2n+1)t)-2t\sin((n+1)t)-2t\sin(nt)\;\;\;  
\end{equation}
If we divide the last equation with $t$ and apply to it the sum-to-product formula we get
$$-4\sin((2n+1)t)-4\sin\frac{(2n+1)t}{2}\cos \frac{t}{2}=0.  $$
When we apply 
sine of double-angle formula
$$-8\sin\frac{(2n+1)t}{2}\cos\frac{(2n+1)t}{2}-4\sin\frac{(2n+1)t}{2}\cos \frac{t}{2}=0,  $$
and
take out the common factor we get 
$$-4\sin\frac{(2n+1)t}{2}\left(2\cos\frac{(2n+1)t}{2}+\cos \frac{t}{2}\right)=0.$$

If $\sin\frac{(2n+1)t}{2}=0$ then 
\begin{equation}\label{c2np1h}
  \cos\frac{(2n+1)t}{2}=\pm 1
\end{equation}
and 
\begin{equation}\label{c2np1}
  \cos((2n+1)t)= 1
\end{equation}

To determine the envelope we need to eliminate $n$ in \eqref{eqEnv1}
$$z=2\cos((2n+1)t)+2\cos((n+1)t)+2\cos(nt)+1 .$$
Using \eqref{c2np1} we get
$$z=2\cdot 1+2\cos((n+1)t)+2\cos(nt)+1 .$$
Now we use sum of cosine to product formula and get 
$$z=3+4\cos\frac{(2n+1)t}{2}\cos\frac{t}{2} $$
Finally we get two envelopes substituting \eqref{c2np1h}
$$ E_1(t)= 3 + 4\cos\frac{t}{2},\;\; E_2(t)= 3 - 4\cos\frac{t}{2}.$$
Zeros of envelopes are
$$3 \pm 4\cos\frac{t}{2}=0\Rightarrow t=2\arccos\frac{\pm 3}{4}.$$

In $[0,2\arccos\frac{ 3}{4}]\cup[2\arccos\frac{- 3}{4},2\pi]$ we have two unimodular roots followed by two nonunimodular roots and vice versa (see Figure \ref{fig:1}. and Figure \ref{fig:2}) as its arguments increase.
In $[2\arccos\frac{3}{4},2\arccos\frac{-3}{4}]$ all roots are unimodular.
It follows that $C(P(x,x^n))\rightarrow \frac{1}{2}\frac{1}{2\pi}2\cdot 2\arccos\frac{ 3}{4}\approx 0.230053456. $
\end{proof}

\subsection{The graph and a crude approximation}

\begin{figure}[t]
  \includegraphics[width=0.9\textwidth, height=0.54\textheight]{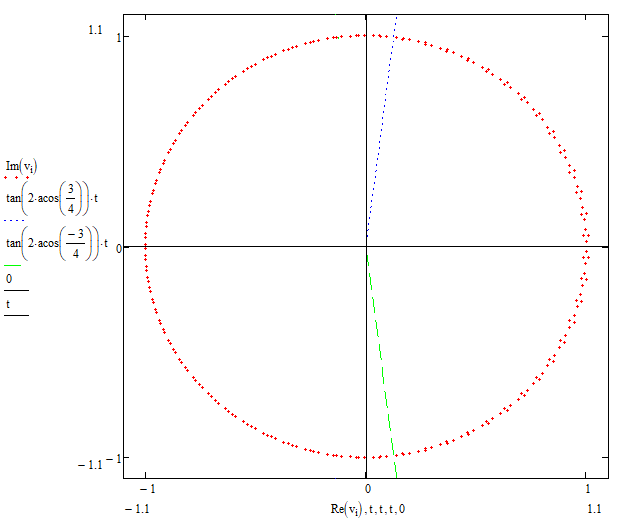}
\caption{Estimation of the number of nonunimodular roots in different sectors of $[0,2\pi]$}
\label{fig:2}       
\end{figure}

If we replace $y$ with $x^{60}$ in \eqref{P23inv} 
we get
$$P^\times_{2,3}(x,x^{60})=x^{242}+x^{182}+x^{181}+x^{121}+x^{61}+x^{60}+1$$
having 28 roots $>1$ (and 28 roots $<1$) in modulus (see Figure \ref{fig:2}) so that $C(P(x,x^{60}))\approx\frac{2\cdot 28}{242}\approx0.2314$ that is close to the exact value $0.230053456\ldots$.


\subsection{The exact values}


If we calculate the exact values of jump points $t_1,t_2,\ldots,t_{m-1}$, defined in the MBM method, we can determine $LC$ exactly using \eqref{APP3}.  
The jump points occur at points for which $P(x, y) = 0$ on the torus $\mathbb{T}^2$. 
\begin{lemma}
For reciprocal $P(x, y)$
with real coefficients, the jump points occur at values of $|x| = 1$ that
are roots of the discriminant polynomial $F(x) = \textrm{disc}_y P(x, y)$. 
\end{lemma}
\begin{proof}
Notice that if $P(e(t), y) = 0$, then also $P(e(-t), \bar{y}) =0$, by complex conjugation, and then $P(e(t), 1/\bar{y}) = 0$,
since $P$ is reciprocal, here $e(t)$ denotes $\exp(2\pi i t)$. A jump point of $\nu(e(t))$ is a value $t = t_j$ such that $|y_k(e(t))| > 1$ for
$t_j<t<t_j + \epsilon$ or $t_j - \epsilon <t<t_j$ and $\lim_{ t \to t_j} |y_k(e(t))| = 1$.
But $1/\overline{y_k(e(t))}$ is a different root of $P(e(t), y)$ for which
$\lim_{ t\to t_j}\left|1/{\overline{y_k(e(t))}}\right| = 1$, hence $P(e(t_j), y) = 0$ 
has a double root and so $\textrm{disc}_y P(e(t_j), y) = 0$.
\end{proof}
For example $\textrm{disc}_y P^\times_{2,3}(x,y)=(4x^2-x+4)(x^2+6x+1)^2$ has unimodular roots $(1\pm 3i\sqrt{7})/8$ so that 
$LC(P^\times_{2,3}(x,y))=\frac{\arctan(3\sqrt{7})}{2\pi}=\frac{\arccos(3/4)}{\pi}$ i.e. the same result claimed in Theorem \ref{MainTh}.

In the second example $\textrm{disc}_y P^\times_{3,1}(x,y)=12x^8+52x^7+60x^6-36x^5+13x^4-36x^3+60x^2+52x+12$ has four unimodular roots so that 
$LC(P^\times_{3,1}(x,y))=$ 

\noindent
$\frac{1}{2\pi}\left(\arctan\left(\frac{0.49793075196076851743}{0.86721679310987954325}\right)+\arctan\left(\frac{0.97538685144561600137}{-0.22050054427825755854}\right)+\pi\right)=$ $0.3683378552178548$ 
is the value presented in 11th row of Table \ref{tab:table2} that we calculated using the MBM method.

In the third example $\textrm{disc}_y P^\times_{3,2}(x,y)=81x^6+66x^5+1007x^4+1788x^3+1007x^2+66x+81$ has two unimodular roots so that 
$LC(P^\times_{3,2}(x,y))=$ 

\noindent
$\frac{4}{5\cdot2\pi}\left(\arctan\left(\frac{0.41804296711468602894}{-0.90842725501052066099}\right)+\pi\right)=$ $0.34508645923655007$ 
is the value presented in 9th row of Table \ref{tab:table2} that we calculated using the MBM method.

\bibliographystyle{crplain}
\bibliography{StankovBib1}{}

\def\bysame{\leavevmode ---------\thinspace}
\makeatletter\if@francais\providecommand{\og}{<<~}\providecommand{\fg}{~>>}
\else\gdef\og{``}\gdef\fg{''}\fi\makeatother
\def\cdrandname{\&}
\providecommand\cdrnumero{no.~}
\providecommand{\cdredsname}{eds.}
\providecommand{\cdredname}{ed.}
\providecommand{\cdrchapname}{chap.}
\providecommand{\cdrmastersthesisname}{Memoir}
\providecommand{\cdrphdthesisname}{PhD Thesis}
\begin{thebibliography}{10}

\bibitem{BCFJ}
P.~Borwein, S.~Choi, R.~Ferguson, J.~Jankauskas, {\og On Littlewood Polynomials
  with Prescribed Number of Zeros Inside the Unit Disk\fg}, \emph{Canadian
  Journal of Mathematics} \textbf{67} (2015), \cdrnumero 3, p.~507–526.

\bibitem{BEFL}
P.~Borwein, T.~Erd{\'e}lyi, R.~Ferguson, R.~Lockhart, {\og On the Zeros of
  Cosine Polynomials: Solution to a Problem of Littlewood\fg}, \emph{Annals of
  Mathematics} \textbf{167} (2008), \cdrnumero 3, p.~1109-1117,
  \url{http://www.jstor.org/stable/40345374}.

\bibitem{BM}
D.~W. Boyd, M.~J. Mossinghoff, {\og {Small Limit Points of Mahler's
  Measure}\fg}, \emph{Experimental Mathematics} \textbf{14} (2005), \cdrnumero
  4, p.~403-414, \url{https://doi.org/10.1080/10586458.2005.10128936},
  \url{https://arxiv.org/abs/https://doi.org/10.1080/10586458.2005.10128936}.

\bibitem{BGMP}
F.~Brunault, A.~Guilloux, M.~Mehrabdollahei, R.~Pengo, {\og Limits of {Mahler}
  measures in multiple variables\fg}, \emph{Annales de l'Institut Fourier}
  \textbf{74} (2024), \cdrnumero 4 (en), p.~1407-1450,
  \url{https://aif.centre-mersenne.org/articles/10.5802/aif.3611/}.

\bibitem{BZ}
F.~Brunault, W.~Zudilin, \emph{Many Variations of Mahler Measures: A Lasting
  Symphony}, Australian Mathematical Society Lecture Series, Cambridge
  University Press, 2020.

\bibitem{DAS2024106087}
M.~K. Das, {\og Distribution of the zeros of polynomials near the unit
  circle\fg}, \emph{Journal of Approximation Theory} \textbf{304} (2024),
  p.~106087,
  \url{https://www.sciencedirect.com/science/article/pii/S0021904524000753}.

\bibitem{Dru}
P.~Drungilas, {\og Unimodular roots of reciprocal Littlewood polynomials\fg},
  \emph{Journal of the Korean Mathematical Society} \textbf{45} (2008),
  \cdrnumero 3, p.~835-840.

\bibitem{Dub}
A.~Dubickas, {\og Mahler measures of Pisot and Salem type numbers\fg},
  \emph{Quaestiones Mathematicae} \textbf{45} (2022), \cdrnumero 9,
  p.~1449-1458, \url{https://doi.org/10.2989/16073606.2021.1948455},
  \url{https://arxiv.org/abs/https://doi.org/10.2989/16073606.2021.1948455}.

\bibitem{F1a1}
V.~Flammang, {\og {The S-measure for algebraic integers having all their
  conjugates in a sector}\fg}, \emph{Rocky Mountain Journal of Mathematics}
  \textbf{50} (2020), \cdrnumero 4, p.~1313 - 1321,
  \url{https://doi.org/10.1216/rmj.2020.50.1313}.

\bibitem{FV}
V.~Flammang, P.~Voutier, {\og {Properties of trinomials of height at least
  $2$}\fg}, \emph{Rocky Mountain Journal of Mathematics} \textbf{52} (2022),
  \cdrnumero 2, p.~507 - 518, \url{https://doi.org/10.1216/rmj.2022.52.507}.

\bibitem{GL}
J.~Gu, M.~N. Lal{\'i}n, {\og The Mahler measure of a three-variable family and
  an application to the Boyd–Lawton formula\fg}, \emph{Research in Number
  Theory} \textbf{7} (2021), p.~1-23,
  \url{https://api.semanticscholar.org/CorpusID:221936118}.

\bibitem{GV}
C.~Guichard, J.-L. Verger-Gaugry, {\og On Salem numbers, expansive polynomials
  and Stieltjes continued fractions\fg}, \emph{Journal de Theorie des Nombres
  de Bordeaux} \textbf{27} (2014), p.~00-00,
  \url{https://api.semanticscholar.org/CorpusID:27873791}.

\bibitem{LNS}
M.~Lalín, S.~S. Nair, {\og An invariant property of Mahler measure\fg},
  \emph{Bulletin of the London Mathematical Society} \textbf{55} (2023),
  \cdrnumero 3, p.~1129-1142,
  \url{https://londmathsoc.onlinelibrary.wiley.com/doi/abs/10.1112/blms.12778},
  \url{https://arxiv.org/abs/https://londmathsoc.onlinelibrary.wiley.com/doi/pdf/10.1112/blms.12778}.

\bibitem{MS}
J.~McKee, C.~Smyth, \emph{Around the Unit Circle: Mahler Measure, Integer
  Matrices and Roots of Unity}, 1st ed. \cdredname, Springer, 2021,
  \url{https://doi.org/10.1007/978-3-030-80031-4}.

\bibitem{Muk}
K.~Mukunda, {\og Littlewood Pisot numbers\fg}, \emph{Journal of Number Theory}
  \textbf{117} (2006), \cdrnumero 1, p.~106 - 121,
  \url{https://www.sciencedirect.com/science/article/pii/S0022314X05001265}.

\bibitem{Pri}
I.~E. Pritsker, {\og {An areal analog of Mahler’s measure}\fg},
  \emph{Illinois Journal of Mathematics} \textbf{52} (2008), \cdrnumero 2,
  p.~347 - 363, \url{https://doi.org/10.1215/ijm/1248355339}.

\bibitem{Sta1}
D.~Stankov, {\og {The number of nonunimodular roots of a reciprocal
  polynomial}\fg}, \emph{Comptes Rendus. Mathématique} \textbf{361} (2023),
  p.~423 - 435,
  \url{https://comptes-rendus.academie-sciences.fr/mathematique/articles/10.5802/crmath.422/}.

\bibitem{VIEIRA2021113169}
R.~Vieira, {\og How to count the number of zeros that a polynomial has on the
  unit circle?\fg}, \emph{Journal of Computational and Applied Mathematics}
  \textbf{384} (2021), p.~113169,
  \url{https://www.sciencedirect.com/science/article/pii/S037704272030460X}.

\bibitem{widder1989advanced}
D.~Widder, \emph{Advanced Calculus}, Dover Books on Mathematics, Dover
  Publications, 1989, \url{https://books.google.rs/books?id=dX5tsSdJtxUC}.

\end{thebibliography}

\end{document}